\documentclass{elsarticle}

\usepackage{amssymb}
\usepackage{amsmath}
\usepackage{amsthm}
\usepackage{cancel}
\usepackage{hyperref}
\usepackage{graphicx}
\usepackage{color}
\usepackage{fmtcount}
\usepackage{tikz}
\usepackage{rotating}
\usepackage{enumitem}
\usepackage{mathtools}
\usepackage{linegoal}

\usetikzlibrary{shapes}
\include{longdiv}
\hypersetup{colorlinks=true, linkcolor=red}

\theoremstyle{definition} 
\newtheorem{Theorem}{Theorem}
\newtheorem{Lemma}[Theorem]{Lemma}
\newtheorem{Corollary}[Theorem]{Corollary}

\newtheorem{Claim}[Theorem]{Claim}
\newtheorem{Open}[Theorem]{Open Problem}

\newcommand{\N}{\mathbb{N}}
\newcommand{\diam}{diam}
\newcommand{\ch}{ch}
\newcommand{\Var}{Var}

\begin{document}

\begin{abstract}
We study a game of pursuit and evasion introduced by Seager in 2012, in which a cop searches the robber from outside the graph, using distance queries. A graph on which the cop wins is called locatable. In her original paper, Seager asked whether there exists a characterisation of the graph property of locatable graphs by either forbidden or forbidden induced subgraphs, both of which we answer in the negative. We then proceed to show that such a characterisation does exist for graphs of diameter at most $ 2 $, stating it explicitly, and note that this is not true for higher diameter. Exploring a different direction of topic, we also start research in the direction of colourability of locatable graphs, we also show that every locatable graph is $ 4 $-colourable, but not necessarily $ 3 $-colourable. 
\end{abstract}

\title{Subgraphs and Colourability of Locatable Graphs}

\author[mem]{Richard A. B. Johnson}

\author[cam]{Sebastian Koch}

\address[mem]{University of Memphis, Memphis TN, USA, rjhnsn25@memphis.edu}
\address[cam]{University of Cambridge, Cambridge, UK, sk629@cam.ac.uk}

\maketitle

\section{Introduction}
\label{SLCIntroductionSection}

Pursuit and evasion games were introduced by Parsons in 1976 \cite{Parsons76} when he considered the problem of locating a lost spelunker in a cave. As a worst-case scenario he considered the situation where the spelunker was actively trying to avoid his pursuers, giving rise to a natural cops and robbers analogy. This was first proposed independently by Nowakowski and Winkler \cite{NowaWink83} and Quilliot \cite{Quilliot78}. An instance of a cops and robbers game consists of a fixed graph $ G $ and a fixed set of players, the cops on one side and the robbers on the other. In the original version, played with one robber -- which is the set-up studied almost exclusively, apart from the short note \cite{HahnMacG06} -- the players first place themselves on the graph and then move alternatingly -- first the cops, followed by the robber. This is repeated until either some cop lands on the same vertex as the robber, resulting in a win for the cops, or, should the cops be unable to achieve this, indefinitely, resulting in a win for the robber. Many variants to the rules have since been proposed, for an excellent survey see \cite{FoThi08}, \cite{Alspach04} as well as the 2011 book `The Game of Cops and Robbers on Graphs' \cite{BookBN}. \\

For a given graph $ G $, the \emph{cop number} $ c(G) $ denotes the minimum number of cops required to catch a robber on $ G $, whereas the \emph{k-capture time} $ capt_k(G) $ denotes the time that $ k $ cops require to capture the robber on $ G $. While the first is a concept that has been studied from the beginning of the existence of the subject, the latter has only recently been introduced in \cite{BGHK09}. The first results on the cop number include the characterisation of graphs of cop number $ 1 $, achieved by Nowakowski and Winkler \cite{NowaWink83}, as well as, among many others. The first results on capture time include results for graphs of cop number $ 1 $ in \cite{BGHK09} and \cite{Gavenciak10} and for grids \cite{Mehrabian10}. Possibly the most important open conjecture in this area is Meyniel's conjecture, claiming that the cop number of any graph on $ n $ vertices is at most $ O(\sqrt(n)) $. Results towards this conjecture have recently been obtained by Bollob\'as, Kun and Leader \cite{BKL13} and \L uczak and Pra\l at \cite{LP10} on random graphs, and by Lu and Peng \cite{LuPeng} as well as Scott and Sudakov \cite{ScoSud10} in the general setting.

In this paper we consider a probing variant of the game which originates from a paper by Seager \cite{Seager12} in 2012. In this 2-player game a robber is assumed to be hiding on a vertex of a simple graph, and a cop wishes to locate him. On the cop's turn she can probe vertices of the graph. When a vertex is probed, the robber truthfully declares the current graph distance from his current location to the probe vertex. Following this the robber makes a move along an edge or remains still, finishing the turn, following which this procedure is repeated. If the cop can successfully identify which vertex the robber is hiding on she wins, in which case we call such a graph \emph{locatable}. If the robber can indefinitely avoid capture then he wins, and we call such a graph \emph{non-locatable}, following the naming conventions in \cite{Seager12}. As the cop can win eventually with probability 1 on any finite graph against a robber who has no knowledge of her future moves, simply by probing random vertices until she hits the current location of the robber. This naturally leads to a different emphasis: we consider the question of whether the cop has a strategy which is guaranteed to win in bounded time, or equivalently whether she can catch an omniscient robber, which not only knows the cop's previous move, but also her following move, thus allowing him to react to it appropriately and avoiding being located if possible. In the variant that we consider we impose the following conditions. While the robber can only move along edges or remain at his current position, the cop can choose any vertex to probe on each of her moves without restriction. Moreover we include the \emph{no-backtrack condition}, which specifies that the robber cannot move to the vertex that the cop has just probed. We note moreover that, following established convention, we assume the robber to be male and the cop to be female. \\

In her original paper \cite{Seager12} Seager showed that the cop wins on all cycles on at least 4 vertices, apart from $ C_5 $ where the robber wins, and on all trees. However she also showed that the robber wins on $ K_4 $, $ K_{3,3} $, and more generally on any graph that contains $ K_4 $ as a subgraph, or $ K_{3,3} $ as an induced subgraph. We will extend some of these results in section \ref{SCLForbiddenSubgraphsDiam2Section} and introduce the new name of \emph{hideout graphs} for graphs with this property. Moreover, Seager asked whether there exists a characterisation of the set of locatable graphs by forbidden subgraphs, which we will show in section \ref{SCLNoForbiddenCharacterisationSection} is not the case. \\

Carraher, Choi, Delcourt, Erickson and West \cite{CCDEW12}, also in 2012, considered the same game but without the no-backtrack condition, and showed that for any graph $ G $ there exists some $ m $, such that the subdivided graph $ G^\frac{1}{m} $ is locatable, or cop-win, as they call it. Let $ m(G) $ be the minimal such. They also obtained upper bounds on the value of $ m(G) $ both in the general case and specifically for complete bipartite graphs, and determined $ m(G) = 2 $ precisely for grids. Following this, Haslegrave and the authors in \cite{HJK14} recently proved the exact values of $ m(G) $ for complete and bipartite graphs, and showed that in these cases $ m(G) $ is a threshold value in the sense that $ G^\frac{1}{m} $ is locatable for every $ m \geq m(G) $. Whether this is the case in general is an open question. \\

The basic notation we use is mostly standard and follows \cite{BookBela}, the main slightly non-standard notation used possibly being the concept of the \emph{closed neighbourhood} $ N[A] = N(A) \cup A $ of a set $ A $, consisting of the union of its neighbourhood and the set itself. Additionally, using one of two similarly established notations, following \cite{BookBela} we will denote the path with $ k $ vertices by $ P_{k} $. We further note that we introduce 
the graph $ K_{2} \oplus E_{3} $ as the graph consisting of disjoint copies of an edge $ K_{2} $ and an indipendent $ 3 $-set $ E_{3} $ with all edges between them added. \\

The remainder of this paper is structured as follows: In section \ref{SCLNoForbiddenCharacterisationSection} we will show that there exists no forbidden subgraph or forbidden induced subgraph characterisation for locatable graphs. In section \ref{SCLForbiddenSubgraphsDiam2Section} we will show that for diameter $ 2 $ graphs on the opposite, such a characterisation exists, and we will present it. In section \ref{SCLColourabilitySection}, we will present the result that every locatable graph is $ 4 $-colourable, and in the final section \ref{SCLOpenProblemsSection} we will close by presenting some related open problems.

\section{Non-existence of Forbidden Subgraph Characterisations}
\label{SCLNoForbiddenCharacterisationSection}

In this section we will answer the question posed by Seager in \cite{Seager12} for the existence of a forbidden subgraph characterisation of the class of locatable graphs. We will show that neither a forbidden subgraph nor a forbidden induced subgraph characterisation exists, thus answering the question in the negative. 

We start by recalling a few basic definitions. A \emph{graph property} is a class of graphs that is closed under isomorphism. A graph property is called \emph{monotone} if it is closed under removal of edges and vertices, and it is called \emph{hereditary} if it is closed under removal of vertices. It is a well-known and easy to check fact that monotonicity (respectively hereditarity) of a graph property is equivalent to the existence of a forbidden (respectively forbidden induced) subgraph characterisation. We will achieve the results on the non-existence of such a characterisation by presenting an example of a non-locatable graph that is an induced subgraph of a locatable graph, thus establishing the following result.


\begin{Theorem}
\label{SCLNotHereditary}
The graph property of being locatable is not a hereditary property.
\end{Theorem}

Clearly, Theorem \ref{SCLNotHereditary} implies 
the same statement for monotonicity. Thus, a forbidden subgraph characterisation for locatable graphs, as asked for by Seager, does not exist, neither for subgraphs nor for induced subgraphs. 

Let us first note that a similar result by example has been established for the robber locating game without the no-backtrack condition by the authors in \cite{HJK14}, where a locatable graph containing an induced $ C_{6} $ subgraph was presented, the latter being non-locatable in the game without no-backtrack. However, with the no-backtrack example $ C_{6} $ becomes locatable as shown in \cite{Seager12}, thus calling for a different example to establish the stronger Theorem \ref{SCLNotHereditary} and disprove the conjecture.

We now continue by giving 
the promised example to prove Theorem \ref{SCLNotHereditary}. We call this graph the \emph{double-net}, which is a triangle with two pending edges on each vertex. Denote by $ a $, $ b $ and $ c $ the vertices on the central triangle, $ T $, of this graph. This is a counterexample because there exist graphs containing it which are non-locatable. For example consider the \emph{rooted double-net}, which is a double net with an additional vertex $ r $ added, the root, which is then connected to the two leaves originating at one of the points of the triangle. We include an illustration of both these graphs in Figure \ref{SCLDoubleNetFigure}. The following lemma will imply the desired result.

\begin{figure}[ht] \centering
  \begin{tikzpicture}
	\tikzstyle{vertex}=[draw, shape=circle, minimum size=5pt, inner sep=0]
	\foreach \x/\y/\label in {-3.0/1.732/A1, -3.5/0.866/A2, -2.5/0.866/A3, -3.5/2.598/A4, -2.5/2.598/A5, -4.5/0.866/A6, -4.0/0.0/A7, -2.0/0.0/A8, -1.5/0.866/A9, 3.0/1.732/B1, 2.5/0.866/B2, 3.5/0.866/B3, 2.5/2.598/B4, 3.5/2.598/B5, 1.5/0.866/B6, 2.0/0.0/B7, 4.0/0.0/B8, 4.5/0.866/B9, 3.0/3.464/B10}
	{\node[vertex] (\label) at (\x, \y) {};
		}
	\foreach \from/\to in {A1/A2, A2/A3, A3/A1, A1/A4, A1/A5, A2/A6, A2/A7, A3/A8, A3/A9, B1/B2, B2/B3, B3/B1, B1/B4, B1/B5, B2/B6, B2/B7, B3/B8, B3/B9, B4/B10, B5/B10}
	{\draw (\from) -- (\to);
		}
	\node at (-3, -1) {Double-net};
	\node at (-3.3, 1.732) {a};
	\node at (-3.4, 0.566) {b};
	\node at (-2.6, 0.566) {c};
	\node at (3, -1) {Rooted double-net};
	\node at (2.7, 1.732) {a};
	\node at (2.6, 0.566) {b};
	\node at (3.4, 0.566) {c};
	\node at (2.7, 3.464) {r};
  \end{tikzpicture}
  \caption{Double-net and Rooted double-net}
  \label{SCLDoubleNetFigure}
\end{figure}
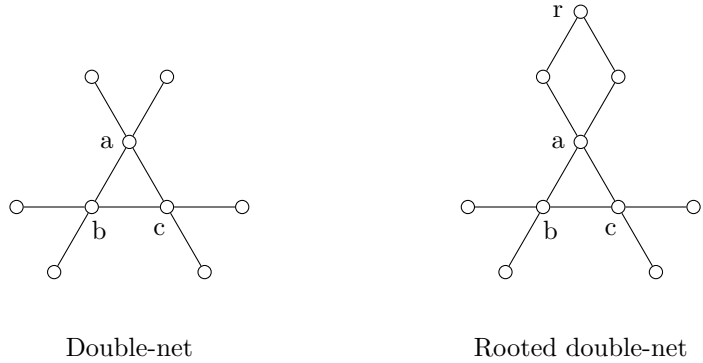


\begin{Lemma}
\label{SCLDoubleNet}
\raisebox{12pt}{
\begin{minipage}[t]{\linewidth}
\begin{enumerate}[leftmargin=*, label=(\roman*)]
\item The double-net is not locatable.\label{SCLDoubleNeti}
\item The rooted double-net is locatable.\label{SCLDoubleNetii}
\end{enumerate}
\end{minipage}}
\end{Lemma}


\begin{proof}
To prove part \ref{SCLDoubleNeti}, we describe a winning strategy for the robber on the double-net, as follows. After every probe and his subsequent move he will claim to be in one of the three following sets which we illustrate in Figure \ref{SCLDoubleNetRobberGoodFigure}:
\begin{enumerate}
\item The central triangle $ T $,
\item A set of two vertices on the triangle plus their respective leafs -- denote this by $ T_{x} $, where $ x \in \{ a, b, c \} $ is the vertex not contained in it,
\item A set denoted by $ T_{x,y} $ with $ x, y \in \{ a, b, c \} $ being distinct, this being the set containing all of the double-net but $ y $ and the two leaves at $ x $. 
\end{enumerate}

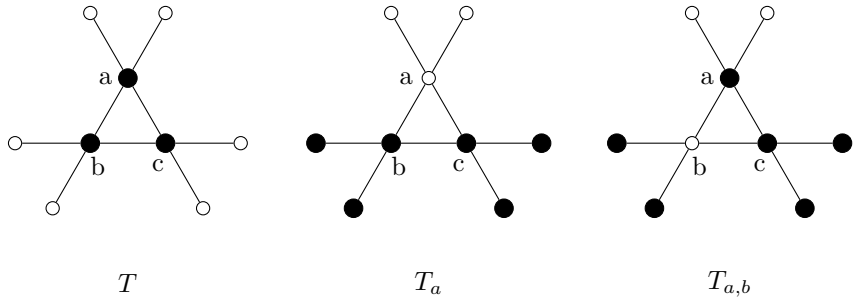
\begin{figure}[ht] \centering
  \begin{tikzpicture}
	\tikzstyle{vertex}=[draw, shape=circle, fill=white, minimum size=5pt, inner sep=0]
	\tikzstyle{blackvertex}=[draw, shape=circle, fill=black, minimum size=7pt, inner sep=0]
	\foreach \x/\y/\name in {-4.5/2.598/A4, -3.5/2.598/A5, -5.5/0.866/A6, -5.0/0.0/A7, -3.0/0.0/A8, -2.5/0.866/A9, 0.0/1.732/B1, -0.5/2.598/B4, 0.5/2.598/B5, 3.5/0.866/C2, 3.5/2.598/C4, 4.5/2.598/C5}
	{\node[vertex] (\name) at (\x, \y) {};
		}
	\foreach \x/\y\name/\label in {-4.0/1.732/A1, -4.5/0.866/A2, -3.5/0.866/A3, -0.5/0.866/B2, 0.5/0.866/B3, -1.5/0.866/B6, -1.0/0.0/B7, 1.0/0.0/B8, 1.5/0.866/B9, 4.0/1.732/C1, 4.5/0.866/C3, 2.5/0.866/C6, 3.0/0.0/C7, 5.0/0.0/C8, 5.5/0.866/C9}
	{\node[blackvertex] (\name) at (\x, \y) {};
	  }
	\foreach \letter in {A, B, C}
	{
		\foreach \from/\to in {1/2, 2/3, 3/1, 1/4, 1/5, 2/6, 2/7, 3/8, 3/9}
		{\draw (\letter\from) -- (\letter\to);
			}
			}
	\node at (-4, -1) {$T$};
	\node at (0, -1) {$T_a$};
	\node at (4, -1) {$T_{a, b}$};
	\node at (-4.3, 1.732) {a};
	\node at (-4.4, 0.566) {b};
	\node at (-3.6, 0.566) {c};
	\node at (-0.3, 1.732) {a};
	\node at (-0.4, 0.566) {b};
	\node at (0.4, 0.566) {c};
	\node at (3.7, 1.732) {a};
	\node at (3.6, 0.566) {b};
	\node at (4.4, 0.566) {c};
  \end{tikzpicture}
  \caption{$T$, $T_a$, $T_{a, b}$ for the Double-Net}
  \label{SCLDoubleNetRobberGoodFigure}
\end{figure}

This can be achieved as follows. When in $ T $, the robber returns the only non-unique distance the cop's probe yields and can without loss of generality move to $ T_{a} $. When in $ T_{a} $, after a probe at $ a $ or the adjacent leaves the robber can move to $ T_{a} $ again, and after a probe at any other vertex he can move to $ T $ or, without loss of generality, $ T_{a,b} $. When in $ T_{a,b} $, the robber can move to $ T_{a} $ without loss of generality, unless the probe is done at $ c $, when he can move to $ T_{b,c} $, or at a leaf of $ c $, when he can move to $ T $, finishing the proof. \\
For the proof of part \ref{SCLDoubleNetii}, we describe a winning strategy for the cop on the rooted double-net. Let her first probe the two leaves of $ b $ and then $ b $ itself, noting that this either locates the robber or drives him out of all those three vertices. Probing in $ r $ next ensures that the vertices corresponding to distances $ 0 $, $ 2 $ and $ 3 $ are unique and locate the robber, whereas distance $ 1 $ and $ 4 $ do not -- returning either of those values will only allow the robber to move into a geodesic path, so he will be located in the next probe. Thus the rooted double-net is locatable.
\end{proof}

This answers the questions for the existence of a characterisation of locatable graphs by forbidden subgraphs or forbidden induced subgraphs in the negative. However, this does not rule out that general results for forbidden induced subgraphs can be found for specific classes of graphs, and we will indeed present one such result in the following section. \\

As a final note in this section, we remark that there are many more examples of graphs that establish the main result of this chapter, as well as the weaker version for monotonicity. For the latter, we would like to end by briefly mentioning a family of such graphs that provides a few interesting insights in this direction. For $ n \in \N $, let the $ n $-sunlet $ Su_{n} $ be the graph consisting of a cycle $ C_{n} $ with one pending edge on every vertex. Any graph containing an induced $ Su_{n} $ for $ n \geq 5 $ is non-locatable, but adding an edge between two adjacent edges makes the sunlet locatable for $ n \geq 6 $. Also taking away any edge of any of these sunlets makes it locatable, making these graphs minimal with this property in some sense. Furthermore, this family shows that there is an infinite number of graphs with the property of making the containing supergraph locatable, and that there is no upper bound in terms of diameter or girth for such graphs. \\

As hinted in the introduction, we will call a graph with the property that any graph containing it is non-locatable a \emph{hideout graph}. Similarly, call a graph such that any graph containing an induced copy of it is non-locatable, but not every graph containing it, a \emph{weak hideout graph}. Thus the sunlets are also interesting in the sense that they are the first examples of weak hideout graphs. 

\section{A Forbidden Subgraph Characterisation for Graphs of Diameter 2}
\label{SCLForbiddenSubgraphsDiam2Section}

In this section we will prove that for graphs of diameter at most $ 2 $ a forbidden subgraph characterisation for locatable graphs does exist, unlike for the general case. We actually prove slightly more, namely that any non-locatable graph of diameter at most $ 2 $ is a hideout graph. We will then go on to characterise precisely which graphs of diameter $ 2 $ are non-locatable, thus presenting this forbidden subgraph characterisation explicitly. Let us open by stating our main result of this section. In this chapter, whenever we talk about containment, we allow any containment, i.e. do not only talk about induced subgraphs. Thus both a forbidden subgraph and an easily deducable induced forbidden subgraph characterisation exists in this case.

\begin{Theorem}
\label{SCLDiam2Hideout}
Any non-locatable graph of diameter at most $ 2 $ is a hideout graph.
\end{Theorem}

Note that the above result is trivial for graphs of diameter at most $ 1 $, as these are the complete graphs and any graph containing a copy of $ K_4 $ is non-locatable (see \cite{Seager12}, included below as Lemma \ref{SCLSeagerLemma2}), so we only deal with the case of diameter exactly $ 2 $. This result is best possible, in the sense that it cannot be extended to diameter $ 3 $ -- not even if we replace hideout graphs by weak hideout graphs, as follows from the double-net example, given in Lemma \ref{SCLDoubleNet} in the previous section. Theorem \ref{SCLDiam2Hideout} is an immediate consequence of the following two results; which in spirit should be understood as one result, but are easier to state as two.

\begin{Lemma}
\label{SCLDiam2Contains}
A non-locatable graph $ G $ of diameter $ 2 $ contains $ C_{5} $, $ K_{4} $, $ K_{2} \oplus E_{3} $ or $ K_{3,3} $ as a subgraph.
\end{Lemma}

\begin{Lemma}
\label{SCLSmallHideouts}
$ C_{5} $, $ K_{4} $, $ K_{2} \oplus E_{3} $ or $ K_{3,3} $ are hideout graphs.
\end{Lemma}

We note here that the combination of these two results yields the aforementioned characterisation of non-locatable diameter-$2$ graphs, and state this again for completeness.

\begin{Theorem}
\label{SCLDiam2Char}
A graph $ G $ of diameter $ 2 $ is non-locatable if and only if it contains any of $ C_{5} $, $ K_{4} $, $ K_{2} \oplus E_{3} $ or $ K_{3,3} $ as a subgraph.
\end{Theorem}

We first turn our attention to Lemma \ref{SCLDiam2Contains}, for which we need some additional results. These will be in the spirit of some of the first results of Seager \cite{Seager12}. Let us start by recalling Propositions 2.3 and 2.4. from her paper using our notions.

\begin{Lemma}
\label{SCLSeagerLemma2}
$ K_{4} $ and $ K_{3,3} $ are weak hideout graphs.
\end{Lemma}

Note however, as the following lemma shows, that $ K_{3,3} $ is not actually the graph causing non-locatability for its supergraphs, in the sense that the proper subgraph $ K_{3,3} $ minus an edge -- which we will refer to as $ K_{3,3}^{-} $ -- is already enough to achieve non-locatability. It moreover also shows that one can drop the assumption of $ K_{3,3} $ being induced in order to cause non-locatability for the supergraph, thus making it not only a weak hideout graph but a hideout graph, which is clearly also the case for $ K_{4} $, where those notions are identical. \\

Despite the fact that not $ K_{3,3} $, but $ K_{3,3}^{-} $ is the minimal non-locatable graph in this context, the latter does not appear in Lemma \ref{SCLDiam2Contains} and thus Theorem \ref{SCLDiam2Char} due to the fact that it has diameter $ 3 $, and, in fact, every graph of diameter $ 2 $ containing it also contains a $ C_{5} $, making its appearance in Lemma \ref{SCLDiam2Contains} unnecessary and making the replacement by $ K_{3,3} $ seem more natural. \\

From now on for the remainder of the paper we will consistently use the letter $ L $ whenever we denote the set of current possible locations of the robber, and the letter $ M $ for the set of vertices he can move towards from $ L $. This will always be made clear when it occurs, but is noted here for convenience.

\begin{Lemma}
\label{SCLK_33Hideout}
$ K_{3,3}^{-} $ is a hideout graph.
\end{Lemma}

\begin{proof}
Let $ G $ be a graph containing a copy of $ K_{3,3}^{-} $, and let $ H $ be a subgraph of $ G $ induced by such a copy of $ K_{3,3}^{-} $. We claim that the robber has a winning strategy where he never leaves $ H $. We will more precisely show that he has such a strategy where he can, after every probe by the cop, move to a set $ M $ that is one of two possible sets of vertices: Either all but at most one vertex of $ H $, or all but the two unique vertices $ a $ and $ b $ of distance $ d_{H}(a,b) = 3 $ in $ H $. \\
We show this by induction, by showing that it is true for step $ n+1 $ if it has been true for step $ n $ and clearly, the statement holds before the first probe. Firstly, for each of the sets $ H \setminus \{ z \} $ for each $ z \in H $ as well as $ H \setminus \{ a,b \} $, a set $ L $ of at least two vertices at the same distance to the probe exists, irrespectively of where the cop probes, thus the robber can not be located. Easy case checking shows that any two vertices in $ H $ have a closed neighbourhood of size at least $ 5 $, hence the induction hypothesis holds if the probe vertex $ p $ was not in $ H $. If it was, note that the distance returned by the robber must be $ 1 $ or $ 2 $ as there is only one pair of vertices, $ a $ and $ b $, at distance $ 3 $ in $ H $. The robber then moves, and can hence be at least on all the vertices that have distance $ 1 $ or $ 2 $ from the probed vertex -- this is not true in general, but it is for this particular graph, as it has no leaves which could be at distance $ 1 $ from $ p $ and not reachable from any vertex at distance $ 2 $ from $ p $. If the probed vertex was not $ a $ or $ b $ this constitutes all five non-probed vertices, and if it was either of $ a $ and $ b $ this constitutes all vertices but those two. Hence the induction claim is true and thus the robber will never be captured, so $ G $ is non-locatable.
\end{proof}

We have thus established Lemma \ref{SCLSmallHideouts} for $ K_{3,3} $ and will now continue by proving it for the two remaining graphs $ C_{5} $ and $ K_{2} \oplus E_{3} $. In these cases, we will actually prove a slightly stronger result, as we will require the details for the upcoming proof of Lemma \ref{SCLDiam2Contains}.

\begin{Lemma}
\label{SCLDiam2MinHideouts}
$ C_{5} $ and $ K_{2} \oplus E_{3} $ are hideout graphs. More precisely, if $ G $ is a graph containing any of these as a subgraph, the robber has a strategy that allows him to move to all but at most one vertex of this $ C_{5} $ or $ K_{2} \oplus E_{3} $ subgraph after every probe.
\end{Lemma}

\begin{proof}
Let $ G $ be as above and let $ H $ be an induced subgraph spanned by a copy of $ C_{5} $ or $ K_{2} \oplus E_{3} $. We will show, again inductively, that after every probe the robber can be on all vertices of $ H $ apart from possibly one, which we will call $ z $ if it exists. Assume this is true after some probe. An elementary checking of cases implies that for any choice of $ z $, the set $ H \setminus \{ z \} $ (as well as, clearly, $ H $ itself) has the property that if this is the set of possible locations of the robber, then whichever vertex $ p $ the cop probes, the robber can choose a set $ L $ of vertices all at the same distance from $ p $ such that $ M = N_{H}[L] \setminus \{ p \} $ contains at least $ 4 $ vertices, implying the result. We leave the details for the reader for the case $ p \in H $, noting only the following. If $ p \notin H $, and we had $ \vert M \vert \leq 3 $, this would imply that $ \vert L \vert = 2 $. Moreover, as $ H $ is connected, and thus $ \vert N_{H}(L) \vert > \vert L \vert $, we could thus further deduce that $ H $ has connectivity $ 1 $. This is clearly false for the graphs we consider.
\end{proof}

We have thus established Lemma \ref{SCLSmallHideouts}. Before now being able to deduce the converse of the latter, we will require a few more auxiliary results. The following definition and the succeeding simple lemma will be useful for these purposes.
We say that a \emph{leaf component} of $ G $ relative to $ v $ is a component of $ G \setminus \{ v \} $. If $ x \in G \setminus \{ v \} $, denote the leaf component of $ G $ relative to $ v $ that contains $ x $ by $ C_{x} $, whenever $ G $ and $ v $ are clear from the context. \\

\begin{Lemma}
\label{SCLLeafComponentsSparse}
Let $ G $ be a graph 
that does not contain any of $ C_{5} $, $ K_{4} $ or $ K_{2} \oplus E_{3} $ as a subgraph. 
Let $ v $ be any vertex. Then no leaf component of $ G $ relative to $ v $ contains more than $ 2 $ edges.
\end{Lemma}

\begin{proof}
Let $ G $ be as above, and let $ C $ be a leaf component of $ G $ (relative to $ v $). $ C $ cannot contain a path of length $ 3 $, as then $ G $ would contain a $ C_{5} $. Moreover $ C $ cannot contain a triangle $ K_{3} $, as then $ G $ would contain a $ K_{4} $. Finally, $ C $ cannot contain a star $ K_{1, 3} $ as then $ G $ would contain a $ K_{2} \oplus E_{3} $. Thus $ C $ cannot contain any three connected edges, which by connectedness of $ C $ implies that $ C $ induces a path of length at most $ 2 $.
\end{proof}

In what follows, we consider graphs on $ n $ vertices that contain a vertex $ v $ which is adjacent to all other vertices in the graph. We now prove Lemma \ref{SCLDiam2Contains} for such graphs of maximum degree $ n-1 $ , from which we will then deduce the full result.

\begin{Lemma}
\label{SCLStarGraphsLocatable}
A graph $ G $ with $ \Delta(G) = n-1 $ that does not contain any of $ C_{5} $, $ K_{4} $ or $ K_{2} \oplus E_{3} $ as a subgraph is locatable.
\end{Lemma}

\begin{proof}
Let $ G $ be as in the statement of the lemma, $ v $ any vertex of degree $ n-1 $. By Lemma \ref{SCLLeafComponentsSparse}, no leaf component of $ G $ relative to $ v $ contains more than two edges. \\
The cop applies the following strategy. She first probes $ v $ and then, in turn, all leaves of $ G $. Following this, she probes all pending triangles in turn, in the sense that she chooses one of its vertices that are not $ v $ and probes it, followed by a probe at $ v $. Finally, the cop then proceeds to probe all leaf components that are paths of length $ 2 $ in a similar way as she did with the triangles, this time by probing the middle vertex and after each such probe probing the central vertex $ v $ again, until the robber returns distance $ 1 $. If the robber returns distance $ 1 $, the cop then probes one of the neighbours of the last probing vertex other than $ v $. \\
It is easy to see by checking the possible results after each probe that this strategy is a winning strategy for the cop: Due to the fact that every second probe is done at $ v $, the robber can never leave the component he started in, and the way the components are probed in turn is chosen such that he is located once his component is probed, which we leave to the reader to verify.
\end{proof}

Note that a graph that has a vertex of degree $ n-1 $ and contains a copy of $ K_{3,3} $ as a subgraph will automatically contain a copy of $ K_{2} \oplus E_{3} $ (and also a copy of $ C_{5} $), as it needs to have at least two edges added in one of the parts of the partition. This explains the we did not need to, and hence did not, include $ K_{3,3} $ as a forbidden subgraph in the above statement. Before we continue to prove Lemma \ref{SCLDiam2Contains}, we note another simple result, that is along the lines for bipartite graphs in \cite{HJK14}, but including the no-backtrack condition.

\begin{Lemma}
\label{SCLK_2mLocatable}
For all $ m \in \N $, the graph $ K_{2,m} $ is locatable.
\end{Lemma}

\begin{proof}
Let $ a $, $ b $ be the vertices in the class of size $ 2 $ and $ c_{1}, \ldots, c_{m} $ the vertices in the class of size $ m $. The following sequence of probes will at some point locate the robber independantly of his moves: Probe at $ a $ at every odd move, and at $ c_{1}, \ldots, c_{m} $ in order at every even move. It is easy to check that this never allows the robber to move to $ a $ or $ b $ without being located, effectively forcing him to stay at some vertex $ c_{k} $ until he is found there.
\end{proof}

We will now use Lemma \ref{SCLStarGraphsLocatable} to deduce the full Lemma \ref{SCLDiam2Contains}.

\begin{proof}[\textup{\textbf{Proof} (of Lemma \ref{SCLDiam2Contains}):}]
Let $ G $ be a graph with $ \diam(G) = 2 $ that does not contain $ C_{5} $, $ K_{4} $, $ K_{2} \oplus E_{3} $ or $ K_{3,3} $ as a subgraph. We will show that $ G $ is locatable by deducing more and more of the possible structure of $ G $ and invoking the previous result where applicable. Let $ v \in G $ be any vertex. Let $ G_{1} = G[N(v)] $ be the graph spanned by the neighbourhood of $ v $ and $ G_{2} = G[G \setminus N[v]] $ the graph spanned by the set of vertices at distance $ 2 $ from $ v $. Without loss of generality assume that $ G_{2} \neq \emptyset $ -- otherwise the result follows immediately from Lemma \ref{SCLStarGraphsLocatable}.

\begin{enumerate}

\item Firstly consider the case that there are no edges within $ G_{2} $. This implies that $ A = G_{2} \cup \{ v \} $ also spans no edges. We will further distinguish three subcases depending on the structure of $ G_{1} $, as follows.

\begin{enumerate}

\item \emph{$ G_{1} $ contains no edges.} Thus $ G $ is bipartite. It must be complete bipartite, as otherwise any two vertices in $ A $ and $ G_{1} $ respectively that are not connected have distance at least $ 3 $. Hence, as $ G $ contains no $ K_{3,3} $, we deduce that either $ G = K_{1,m} $ or $ G = K_{2,m} $ for some $ m \in \N $, all of which are known to be locatable -- for example by Lemmas \ref{SCLStarGraphsLocatable} and \ref{SCLK_2mLocatable}, respectively.

\item \emph{$ G_{1} $ is connected.} As $ v $ is adjacent to all of $ G_{1} $ by definition, it follows by Lemma \ref{SCLLeafComponentsSparse} that $ G_{1} $ spans a path $ P $ of length at most $ 2 $. \\
The case $ P = P_{0} $ being a single vertex leads back to the case that $ G $ has maximum degree $ n-1 $ and is thus locatable by Lemma \ref{SCLStarGraphsLocatable}. \\
In the case $ P = P_{1} $, let $ G_{1} = \{ x, y \} $. If either of $ N(x) $ and $ N(y) $ contains the other, without loss of generality $ N(x) \supseteq N(y) $, then $ G $ again has a vertex of degree $ n-1 $, which this time is $ x $, so Lemma \ref{SCLStarGraphsLocatable} again applies and implies locatability. So assume now that there exist $ s, t \in G_{2} $ with $ s \sim x $, $ s \nsim y $, $ t \sim y $ and $ t \nsim x $. But then, as $ G_2 $ spans no edges, $ d(s,t) \geq 3 $, a contradiction. \\
Thus assume now that $ P = P_{2} $. Let $ x, y $ be the endpoints of $ P $ and $ z $ the centre. If all vertices in $ G_{2} $ are connected to $ z $, $ G $ again has a vertex of degree $ n-1 $, so assume that this is not the case. Moreover any vertex $ w $ from $ G_{2} $ that is connected to $ x $ but not to $ z $ must also be connected to $ y $, as otherwise we have $ d(w,y) \geq 3 $, and vice versa. But then the vertices $ v, z, x, w, y $, in that order, form a $ C_{5} $, a contradiction.

\item \emph{$ G_{1} $ is not connected but contains some edges.} Let $ w \in G_{2} $. Note that $ w $ has to have a neighbour in every component of $ G_{1} $, as if it has no neighbour in a component $ C $, the distance between $ w $ and $ C $ is at least $ 3 $. As $ G_{1} $ contains at least one edge, this implies that there exist $ x, y, z \in G_{1} $ such that $ w \sim x $, $ w \sim y $, $ x \sim z $. Thus, again, the vertices $ v, z, x, w, y $, in that order, form a $ C_{5} $, a contradiction. This finishes the case that $ G_{2} $ contains no edges.

\end{enumerate}

\item Now assume that there are edges in $ G_{2} $. Firstly, if there exists an edge $ (w,x) $ in $ G_{2} $ such that its endpoints have distinct neighbours $ y, z $ in $ G_{1} $, then the vertices $ v, z, x, w, y $ form a $ C_{5} $ in $ G $, contradiction. Thus we can now assume that for any edge in $ G_{2} $, its endpoints have precisely one neighbour in $ G_{1} $, which is the same for both. But if any two edges in $ G_{2} $ have a different vertex they are connected to, the endpoints of these edges would be at least distance $ 3 $ apart, a contradiction. So all edges are connected to the same vertex $ w \in G_{1} $. Moreover, if any $ z \in G_{1} $ is not connected to $ w $, we have $ d(w,G_{2}) \geq 3 $, again a contradiction. Thus, $ w $ has degree $ n-1 $, and thus $ G $ is locatable by Lemma \ref{SCLStarGraphsLocatable}. This finishes the proof and establishes the necessity of the condition in Theorem \ref{SCLDiam2Char}.

\end{enumerate}

\end{proof}

For the sake of completeness, we note that the observations made in this section also allow us to give a complete characterisation of all locatable graphs of diameter at most $ 2 $. This can easily be extracted from the previous proofs, and we will finish the section with this characterisation.

\begin{Corollary}
\label{SCLDiam2LocStructure}
Let $ G $ be a locatable graph of diameter $ 2 $. Then either $ G $ is spanned by a star with centre $ v $, such that for any leaf component $ C $ of $ G $ (with respect to $ v $), $ C \in \{ P_{0}, P_{1}, P_{2} \} $, or $ G = K_{2,m} $ for some $ m \in \N $.
\end{Corollary}

\section{Non-locatability Criteria and Colourability}
\label{SCLColourabilitySection}

We will next present a result establishing a very large family of non-locatable graphs limited purely by a restriction on small degrees of its vertices and by the structure of certain subgraphs. The main result of this section, in the sense that it motivated and triggered all the work that is presented here, is the following, for the deduction of which a partial result regarding the aforementioned non-locatable graphs is sufficient. 

\begin{Theorem}
\label{SCLColourability}
Every locatable graph is $ 4 $-colourable.
\end{Theorem}

As we will show by means of an example at the end of this section, this result is best possible in the sense that there exist locatable graphs that can not be coloured with $ 3 $ colours. Locatable graphs that are not $ 3 $-colourable, however, have very specific substructures, and we make a first step towards a characterisation of those graphs in Theorem \ref{SCLMinDeg3HideoutCondition}. For the results we prove on the way, proofs will often be similar, and we will thus start this section by roughly describing the general outline of most of the non-locatability proofs we are going to present. \\

We will very often invoke the following observation, which we will refer to as the \emph{subset argument}, in order to prove non-locatability, by showing that even if the robber restricts himself to some subgraph of the graph and never leaves it, the cop cannot locate him inside it.

\begin{Lemma}
\label{SCLSubsetArgument}
Let $ G $ be a graph and $ A, B $ subsets of $ V(G) $ such that $ A \subseteq B $. Then if the robber wins from starting in $ A $, he also wins from starting in $ B $.
\end{Lemma}



For example, whenever we want to show that a graph containing some graph as an (induced) subgraph is not locatable, we will assume that the robber restricts himself to this subgraph for the whole game. After having established that the robber restricts himself to a given sugraph $ H $, there are some very similar ideas how to show non-locatability in this context which we will use regularly. One way, which we have encountered earlier, is to show that there is a family $ \mathcal{A} $ of subgraphs of $ H $ such that whenever he is in one of those sets, whereever the cop probes, the robber can then move in such a way that he can again be in a superset of a graph $ F \in \mathcal{A} $. The existence of such a family is sufficient for non-locatability by Lemma \ref{SCLSubsetArgument}. Proofs of this type have also appeared in \cite{HJK14} by the authors and J. Haslegrave. This strategy gives fully constructive proofs which depict closely what is going on when the actual game is played. Even though the subgraphs in $ \mathcal{A} $ might never appear when both sides play their optimal strategy, all sets appearing are supersets of sets in $ \mathcal{A} $, so such a proof contains the essence of what is going on on these graphs. As is intuitive and as we have seen, these proofs often have an inductive nature. \\

Proofs along these lines will not appear in this section -- their disadvantage is that due to their very specific nature, they can only capture small families of very similar graphs at a time. A more general approach is to show that after every probe of the cop, although we can not guarantee some given subgraph for the robber, we can guarantee him a certain amount of control over the choice of vertices where he could be. At the simplest level, this just means that the robber will always be able to return a distance such that there are two possible vertices for him. In this context, we draw the reader's attention to Lemma \ref{SCLChoiceVariance}, which will do precisely that. This approach frees us from needing information about specific subgraphs that are required to appear and so allows us to make stronger statements. \\

One such statement is the following theorem, which forbids a large group of graphs as subgraphs in locatable graphs and allows us to deduce Theorem \ref{SCLColourability}.

\begin{Theorem}
\label{SCLMinDeg4Hideouts}
Let $ H $ be a graph with minimal degree $ \delta(H) \geq 4 $. Then $ H $ is a hideout graph.
\end{Theorem}

Before presenting this proof, we first establish an idea and some notation attached to it which will be useful both for the proof of this result and the upcoming similar but more complicated result Theorem \ref{SCLMinDeg3HideoutCondition}. Moreover, we also hope that the concept will prove useful in the development of further results along the lines of what we are about to present in this section. \\

For what follows, let $ H $ be a graph and $ G $ a connected graph containing $ H $ as a subgraph. For $ v \in G $, let $ D_{v}(H) = \{ d \in \N \mid \exists \ x \in H: d(v,x) = d \} $ be the set of possible distances between $ v $ and vertices of $ H $, and define the \emph{distance variance} of $ H $ in $ G $ as $ \Var_{G}{H} = \underset{v \in G}{\max} \vert D_{v}(H) \vert $, the maximum over the sizes of those for all vertices of $ G $. Furthermore, define the \emph{choice variance} of $ H $ in $ G $ as $ \ch_{G}(H) = \vert H \vert - \Var_{G}(H) $. If the underlying graph $ G $ is clear from the context, we will often drop the indices. \\

We make the following easy observation for the case that $ H $ is connected, which allows us to find bounds on the distance variance from bounds on the diameter. Despite not being very deep, it will turn out to be very useful, as the situations in which we work with distance variance later are only ever complicated to handle for connected $ H $ anyway. 

\begin{Lemma}
\label{SCLVarianceBiggerThanDiameter}
Let $ G $ be a graph, $ H $ a connected subgraph. Then $ \Var_{G}(H) \leq \diam(H) + 1 $.
\end{Lemma}

\begin{proof}
The choice variance (by means of bounding the values it is maximised over) is at most the number of vertices on a longest path in $ H $, which is $ \diam(H) + 1 $.
\end{proof}

Let us understand why the concepts of distance variance and choice variance will be useful for us. The distance variance of $ H $ in $ G $ is an upper bound on the number of vertices of $ H $ on any shortest path between any two vertices of $ G $ (or, de facto, of $ H $). As such, it represents the maximal possible number of distinct distances between a probing position and the vertices of $ H $ the cop can achieve in his next move. \\

With this knowledge, it is also easy to make sense of the notion of choice variance. Clearly, we always have $ \Var_{G}(H) \leq \vert H \vert $, and thus $ \ch_{G}(H)\geq 0 $. Whenever $ \ch_{G}(H) = 0 $, equality holds, i.e. $ \Var_{G}(H) = \vert H \vert $, which means that there is some vertex $ v \in G $ such that any distance from $ v $ to a vertex in $ H $ is unique. This implies that if it is known that the robber is on $ H $, probing at $ v $ will locate the robber. Conversely, whenever $ \ch_{G}(H) > 0 $, the knowledge that the robber is in $ H $ will not allow the cop to immediately locate him, as there is always a choice for the robber to return a distance which is not unique between $ v $ and vertices of $ H $. The bigger the choice variance, the more choice the robber has -- either he can choose between a lot of non-unique distances, or he can choose some distance such that there are many vertices in $ H $ at this distance, or anything inbetween. The most obvious, but not the only application, is the following, which we will refer to as the \emph{sufficient choice lemma}. The proof is straightforward and along the lines of the discussion in this paragraph.


\begin{Lemma}
\label{SCLChoiceVariance}
Let $ G $ be a graph such that the robber has a strategy where after every move, the set of his possible locations, $ H $, satisfies $ \ch_{G}(H) \geq 1 $. Then $ G $ is non-locatable. 
\end{Lemma}

\begin{proof}
By the arguments made above, $ \ch_{G}(H) \geq 1 $ implies that the robber has two vertices to choose in $ H $ whereever the cop probes, and thus by induction he will never be caught.
\end{proof}

Having now established the concept of choice variance, we will continue to prove Theorem \ref{SCLMinDeg4Hideouts}, where this concept will come to its first use.

\begin{proof}(of Theorem \ref{SCLMinDeg4Hideouts})
Let $ H $ be as above, and let $ G $ be a graph containing a copy of $ H $ as a subgraph. We will show that the robber has a winning strategy where he never leaves $ H $ by showing that after every probe the robber can choose to return a distance such that he has at least two possible locations at this distance from the probing vertex, which implies non-locatability by Lemma \ref{SCLChoiceVariance}. We prove this inductively, and it is clearly true before the game started, so assume after some probe of the cop at $ p \in G $ the robber returns distance $ d $ and there are $ x, y \in H $ such that $ d(p,x) = d = d(p,y) $. Then the robber can move and thus claim to be in the graph $ M $ spanned by $ \Gamma_{H}[x,y] \setminus \{ p \} $. Let $ D = d(x,y) $ be the distance between the two locations. Note that it suffices to show that whatever the configuration, for all common neighbours $ z $ of $ x $ and $ y $, $ M_{z} = \Gamma_{H}[x,y] \setminus \{ z \} $ satisfies $ \ch(M_{z}) \geq 2 $, where $ z $ is taking the role of potential probing positions $ p $, as the choice variance with $ z $ included can not be smaller. The sufficient choice argument, Lemma \ref{SCLChoiceVariance}, will then imply the induction hypothesis for the next step. At this point, we would like to remind the reader of Lemma \ref{SCLVarianceBiggerThanDiameter} to help us get easy bounds on the distance variance. If $ x $ and $ y $ do not have a common neighbour, the same needs to be shown for $ M $ itself. We will distinguish $ 3 $ cases for the distance $ D $.

\begin{enumerate}

\item Firstly assume $ D = 1 $. Then $ \diam(M) \leq 3 $. If $ x $ and $ y $ do not have a common neighbour, $ \vert M \vert \geq 7 $ and sufficient choice applies. Thus assume now there is such a neighbour $ z $. Still, the sufficient choice argument implies right away that the induction statement is true -- noting that when $ \vert M_{z} \vert = 4 $ we have $ \diam(M_{z}) = 2 $. 

\item Next, let $ D = 2 $. This implies that we have $ \Gamma_{H}(x) \cap \Gamma_{H}(y) \neq \varnothing $, but $ x \nsim y $. Let $ z $ be a common neighbour. We thus have $ \diam(M_{z}) \leq 4 $ and the sufficient choice argument applies if $ \vert M_{z} \vert \geq 6 $. Otherwise, $ \vert M_{z} \vert \leq 5 $, i.e. $ M_{z} $ must induce a $ K_{2,3} $, where sufficient choice clearly, again, takes effect.

\item Finally, $ D \geq 3 $. In this case, $ x $ and $ y $ are sufficiently far away to yield $ \Gamma_{H}(x) \cap \Gamma_{H}(y) = \varnothing $, and thus $ \vert M \vert \geq 10 $. Moreover, the diameter of each of the two neighbourhoods is at most $ 2 $, so $ \Var(M) \leq 6 $, implying $ \ch(M) \geq 4 $. Hence we are done by sufficient choice, finishing the proof.

\end{enumerate}

\end{proof}

We will now present the proof of Theorem \ref{SCLColourability}, which follows easily.

\begin{proof} (of Theorem \ref{SCLColourability})
Let $ G $ be a graph that is not $ 4 $-colourable, and let $ H $ be a minimal not $ 4 $-colourable subgraph of $ G $. Clearly, $ H $ has minimal degree at least $ 4 $. Otherwise, take a vertex $ x $ of lesser degree, remove it, colour the resulting graph in $ 4 $ colours, which is possible by minimality, and colour $ x $ in a colour none of its less than $ 4 $ neighbours have. An application of Theorem \ref{SCLMinDeg3HideoutCondition} now implies that $ G $ is not locatable, as required.
\end{proof}

Next, we will present a result that further generalises the above, but requires a substantially larger amount of work. The core idea of this proof is already contained in the proof of Theorem \ref{SCLMinDeg4Hideouts}, but, due to the large number of possible cases arising, is considerably longer and more technical. The motivation behind what follows was the failed attempt to prove that all locatable graphs are $ 3 $-colourable, which, as we will see from an example shortly, turns out to be false. The proof of the following result thus contains most of the techniques that were used to attempt and achieve this result. Before we now turn our attention to this, we will define a few specific small graphs that will appear in the proof. \\

We call the graph obtained by deleting one edge of a $ K_{4} $ the \emph{diamond graph}. Moreover, call the degree $ 3 $ vertices \emph{girdle vertices} of a diamond graph and the degree $ 2 $ vertices \emph{tip vertices}. Let the \emph{kite graph} be a diamond with an additional pending edge on one of the tip vertices, whereas the \emph{dart graph} is a diamond graph with an additional pending edge on one of the girdle vertices. The \emph{double dart graph} is the dart graph with the edge connecting the two girdle vertices of the contained diamond being subdivided once. Finally, the \emph{bull graph} is a triangle with two pending edges, and the \emph{watch graph} is a $ C_{4} $ with two pending edges at non-adjacent vertices. These graphs are shown in Figure \ref{SCLSmallGraphDefinitionsFigure} below.

\begin{figure}[ht] \centering
  \begin{tikzpicture}
	\tikzstyle{vertex}=[draw, shape=circle, minimum size=5pt, inner sep=0]
	\foreach \x/\y/\label in {-4.0/2.5/A1, -4.5/2.0/A2, -4.0/1.5/A3, -3.5/2.0/A4, 0.0/2.5/B1, -0.5/2.0/B2, 0.0/1.5/B3, 0.5/2.0/B4, 0/0.79/B5, 3.5/2.0/C1, 4.0/2.5/C2, 4.5/2.0/C3, 4.0/1.5/C4, 2.79/2.0/C5, -4.0/-1.5/D1, -4.5/-2.0/D2, -4.0/-2.5/D3, -3.5/-2.0/D4, -5.21/-2.0/D5, -4.0/-2.0/D6, -0.5/-2.0/E1, 0.0/-2.87/E2, 0.5/-2.0/E3, -1.21/-1.21/E4, 1.21/-1.21/E5, 4.0/-2.5/F1, 3.5/-2.0/F2, 4.0/-1.5/F3, 4.5/-2.0/F4, 2.79/-2.0/F5, 5.21/-2.0/F6}
	{\node[vertex] (\label) at (\x, \y) {};
		}
	\foreach \from/\to in {A1/A2, A2/A3, A3/A4, A4/A1, A2/A4, B1/B2, B2/B3, B3/B4, B4/B1, B2/B4, B3/B5, C1/C2, C2/C3, C3/C4, C4/C1, C1/C3, C1/C5, D1/D2, D2/D3, D3/D4, D4/D1, D2/D5, D2/D6, D4/D6, E1/E2, E2/E3, E3/E1, E1/E4, E3/E5, F1/F2, F2/F3, F3/F4, F4/F1, F2/F5, F4/F6}
	{\draw (\from) -- (\to);
		}
	\node at (-4, 0.3) {Diamond};
	\node at (0, 0.3) {Kite};
	\node at (4, 0.3) {Dart};
	\node at (-4, -3.7) {Double-Dart};
	\node at (0, -3.7) {Bull};
	\node at (4, -3.7) {Watch};
  \end{tikzpicture}
  \caption{Graphs required in the proof of Theorem \ref{SCLMinDeg3HideoutCondition}}
  \label{SCLSmallGraphDefinitionsFigure}
\end{figure}
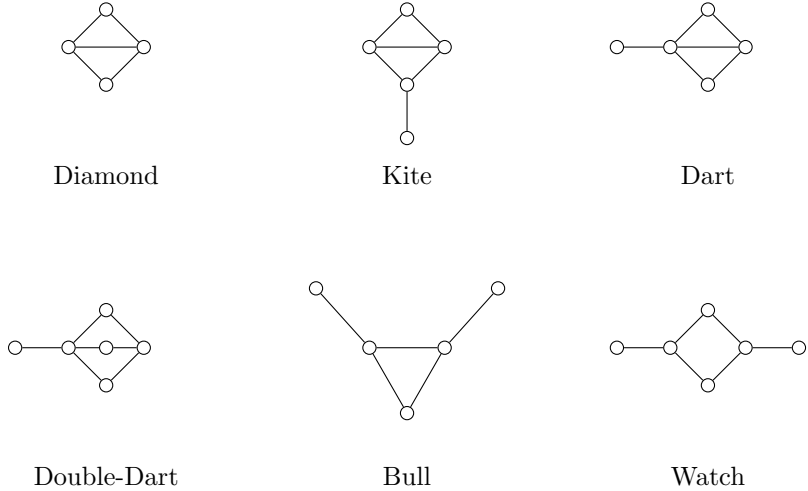




With these in mind, we will now prove the following. 

\begin{Theorem}
\label{SCLMinDeg3HideoutCondition}
Let $ H $ be a graph with minimal degree $ \delta(H) \geq 3 $ such that in every diamond-subgraph, either both girdle or both tip vertices have degree at least $ 4 $. Then $ H $ is a hideout graph.
\end{Theorem}

\begin{proof}
Let $ H $ be as above, and let $ G $ be a graph containing a copy of $ H $ as a subgraph. We will show that the robber has a winning strategy where he never leaves $ H $. This will again be done inductively by showing that after every probe of the cop, the robber cannot be located. More specifically, we will show the following statement. \\

Given a probed vertex $ p $, the robber can choose a distance $ d $ and thus a set $ L $ of possible locations in $ H $, such that $ p $, $ d $ and $ L $  satisfy one of the following conditions.

\begin{enumerate}
\item \label{SCLMinDeg3RobberStrati} $ \vert L \vert \geq 3 $, or
\item \label{SCLMinDeg3RobberStratii} $ \vert L \vert = 2 $ and either $ p \notin H $ or $ d \geq 2 $, or
\item \label{SCLMinDeg3RobberStratiii} $ \vert L \vert = 2 $, $p \in H $, $ d = 1 $ and $ \Gamma_{H}[L] $ induces a diamond graph.
\end{enumerate}

This will imply that the robber can then not be located and can thus move. Clearly, the above statement is true before the game starts. Let us hence turn our attention directly to the induction step. The idea is that we would like to apply Lemma \ref{SCLChoiceVariance} again, as we did in Theorem \ref{SCLMinDeg3HideoutCondition}. However, this is not sufficient in this case, as it would not yield a strong enough assumption to successfully complete the inductive step. Instead, we will use the two following two lemmas, which are clearly along the lines of Lemma \ref{SCLChoiceVariance}, the sufficient choice argument. Despite them being very obvious, we state them for future reference, and will start with some notation. For the remainder of the proof, as done before, denote by $ M $ the set of vertices to which the robber can move after having been in $ L $ -- that is, $ M = \Gamma_{H}[L] \setminus \{ p \} $ due to the no-backtracking condition. When referring to the probe following $ p $, we will use $ p^{+} $ for the probing vertex, $ d^{+} $ for the distance choosen by the robber and $ L^{+} $ for his resulting set of possibilities. \\

\begin{Claim}
\label{SCLSufficientChoice}
If $ \ch_{H}(M) \geq 2 $, for all probing vertices $ p^{+} $ the robber can choose $ d^{+} $ and thus $ L^{+} $ such that the induction statement holds for the next step.
\end{Claim}

\begin{proof}
Let $ M $ be as above. Whichever vertex $ p^{+} $ the cop probes next, either two distances to $ L $ appear at least twice, and thus one distance which is not $ 1 $, implying Situation \ref{SCLMinDeg3RobberStratii} in the induction statement, or some distance appears at least three times -- this implies Situation \ref{SCLMinDeg3RobberStrati}.
\end{proof}

This lemma will be used frequently in this proof, and we will refer to it by the \emph{strong sufficient choice argument}. Another argument is the very similar \emph{structural choice argument}, which is the following. 

\begin{Claim}
\label{SCLStructuralChoice}
If $ \ch_{H}(M) = 1 $,  for all probing vertices $ p^{+} $ the robber can choose $ d^{+} $ and thus $ L^{+} $ such that the induction statement holds for the next step.
\end{Claim}

\begin{proof}
Let $ M $ be as above. Then $ M $ is not a geodesic path, but there exists some vertex $ w \in H $ such that the $ m = \vert M \vert $ vertices of $ M $ are at $ m - 1 $ different distances from $ w $ -- so all but one vertex of $ M $ lie on some geodesic path in $ H $. The fact that $ H $ has minimum degree $ 3 $ restricts the possibilities for $ M $ hugely, and a quick thought shows that $ M $ must be one of the seven graphs shown in Figure \ref{SCLMCandidatesFigure}.

\begin{figure}[h!] \centering
  \begin{tikzpicture}
	\tikzstyle{vertex}=[draw, shape=circle, minimum size=5pt, inner sep=0]
	\tikzstyle{blackvertex}=[draw, shape=circle, fill=black, minimum size=5pt, inner sep=0]
	\foreach \x\y\name in {0.13/4.5/A1, 1.0/6.0/A2, 1.87/4.5/A4, 3.3/4.3/B1, 4.0/6.0/B2, 4.8/4.3/B4, 6.3/5.0/C1, 7.7/5.0/C4, 9.1/5.0/D1, 9.8/5.7/D2, 11.5/5.0/D5, 5.9/1.13/E2, 4.69/2.79/E4, 7.11/2.79/E5, 0.3/2.0/F1, 2.0/2.7/F3, 2.0/1.3/F4, 3.7/2.0/F6, 8.1/2.0/G1, 9.8/2.7/G3, 9.8/1.3/G4, 11.5/2.0/G6}
	{
	 \node[vertex] (\name) at (\x, \y) {};
	 }
	 
	\foreach \x\y\name in {1.0/5.0/A3, 4.0/5.0/B3, 7.0/5.7/C2, 7.0/4.3/C3, 9.8/4.3/D3, 10.5/5.0/D4, 5.4/2.0/E1, 6.4/2.0/E3, 1.3/2.0/F2, 2.7/2.0/F5, 9.1/2.0/G2, 10.5/2.0/G5}
	{
	\node[blackvertex] (\name) at (\x, \y) {};
	 }	 
	 
	\foreach \from/\to in {A1/A3, A2/A3, A3/A4, B1/B3, B2/B3, B1/B4, B3/B4, C1/C2, C1/C3, C2/C3, C2/C4, C3/C4, D1/D2, D1/D3, D2/D3, D2/D4, D3/D4, D4/D5, E1/E2, E2/E3, E1/E3, E1/E4, E3/E5, F1/F2, F2/F3, F2/F4, F3/F5, F4/F5, F5/F6, G1/G2, G2/G3, G2/G4, G3/G4, G3/G5, G4/G5, G5/G6}
	{\draw (\from) -- (\to);
		}
	\node at (0.0, 6.0) {A};
	\node at (3.0, 6.0) {B};
	\node at (6.0, 6.0) {C};
	\node at (9.0, 6.0) {D};
	\node at (1.0, 2.5) {E};
	\node at (5.9, 2.5) {F};
	\node at (8.8, 2.5) {G};
  \end{tikzpicture}
  \caption{Candidates for $M$. The minimal set $ L $ of positions leading to this shape (modulo automorphisms) is highlighted.}
  \label{SCLMCandidatesFigure}
\end{figure}
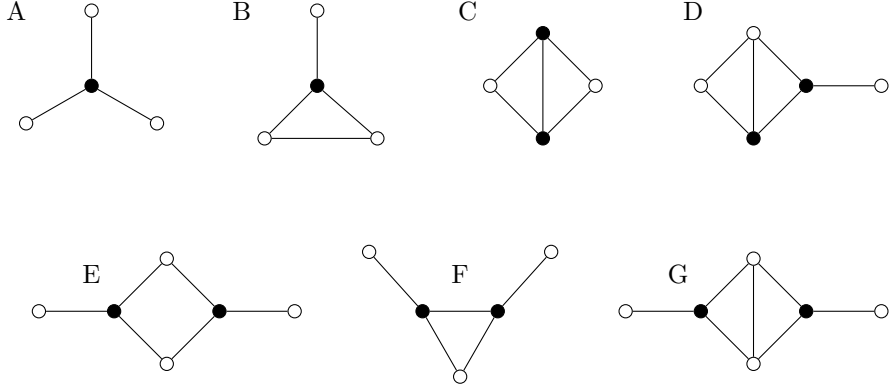

Clearly, $ M $ can not be one of the two graphs $ A $ and $ B $, as they can only originate from a set $ L $ of size $ 1 $, which would have meant that the robber got caught after the probe at $ p $, contraditing the assumption. If $ M $ is a diamond, $ C $, we arrive at case \ref{SCLMinDeg3RobberStratiii} in the induction hypothesis -- note that as $ M \subseteq \Gamma_{H}[L] $, we can assume equality here by the subset argument. $ M $ can moreover not be a kite graph, $ D $, as this would imply that the tip vertex of degree $ 3 $ and at least one girdle vertex of the contained diamond would have been in $ L $. But due to the condition on diamond subgraphs, at least one of them would have to have degree $ 4 $. Finally, if $ M $ is one of the three graphs $ E $, $ F $ and $ G $ in the bottom line, a little case analysis shows that whatever the next probing vertex $ p^{+} $ is, the robber can choose to return a distance such that on of Situations \ref{SCLMinDeg3RobberStrati} and \ref{SCLMinDeg3RobberStratii} from the induction hypothesis becomes true.
\end{proof}


Equipped with this machinery, we now return to establishing the induction step. Assume thus that the induction hypothesis is true after some probe of the cop at $ p \in G $, where the robber returned distance $ d $ and then claims to be in a set $ L $ meeting one of the above conditions. We will now show that the statement is also true after the next probe at $ p^{+} \in G $, by checking the three situations in turn.

\begin{enumerate}

\item We start with situation \ref{SCLMinDeg3RobberStrati}, i.e. $ \vert L \vert \geq 3 $, and will assume equality in the assumption, i.e. $ \vert L \vert = 3 $, by the subset argument. By noting that if $ p \notin H $ or $ d \geq 2 $, the robber can discard one of his potential locations, we may assume that $ p \in H $, $ d = 1 $, the remaining cases being covered by situation \ref{SCLMinDeg3RobberStratii} and thus by the following case. Let the three vertices in $ L $ be $ x $, $ y $ and $ z $. The robber then moves and could thus be in $ M $. Let $ M^{'} = H[M \cup \{ p \}] $. If $ L $ spans a triangle, $ M^{'} $, and thus $ G $, contains a $ K_{4} $, making $ G $ non-locatable. If $ L $ spans at most one edge, we either have $ \vert M \vert \geq 7 $ and $ \Var(M) = \diam(M^{'}) \leq 4 $, allowing the application of the strong sufficient choice argument, or $ H $ contains a $ C_{5} $, which implies non-locatability right away by Lemma \ref{SCLSmallHideouts}.

If $ L $ spans precisely two edges, without loss of generality $ x \sim z \sim y $, it is easy to see that $ p $, $ x $, $ y $ and $ z $ span a diamond. Thus, by the assumption on $ H $, at least two of those four vertices have degree at least $ 4 $. If any of the neighbours of these vertices outside the diamond coincide such that $ H $ contains a $ C_{5} $, non-locatability follows again. If this is not the case, the only neighbours that may still coincide are the potential additional neighbours of $ p $ and $ z $, in which case we get $ \vert M \vert \geq 7 $ by inspection, as well as once again $ \diam(M) \leq 4 $, allowing another application of the strong sufficient choice argument. \\

\item Next, we work under the assumption of situation \ref{SCLMinDeg3RobberStratii} that $ \vert L \vert = 2 $ and that either $ p \notin H $ or $ d \geq 2 $. Hence there are $ x, y \in H $ such that $ d(p,x) = d = d(p,y) $. Let $ D = d_{H}(x,y) $ be the distance between those vertices in $ H $. We will now further distinguish the cases $ D = 1 $, $ D = 2 $ and $ D \geq 3 $. Note that due to the assumption, the robber can be anywhere in $ V(M) = \Gamma_{H}(x) \cup \Gamma_{H}(y) $ after his next move. A careful reader will recognise this part of the proof as being very similar to the proof of Theorem \ref{SCLMinDeg4Hideouts}. \\

\begin{enumerate}

\item Assume $ D = 1 $ first. Then $ \diam(M) \leq 3 $ and thus whenever $ \vert M \vert \geq 6 $, the strong sufficient choice argument implies that the induction statement is true in this case. Thus assume $ \vert M \vert \leq 5 $ from now on. Then $ M $ is either spanned by a dart graph, a bull graph or a diamond graph. Firstly, we note that for some of those graphs, the induction hypothesis also holds by strong sufficient choice, which is the case for the dart graph and for some of the cases where $ M $ is spanned by a bull graph. The bull has $ 5 $ vertices and diameter $ 3 $ with two unique vertices, the bull's `horns', at maximal distance. So if there are edges inside $ M $ shortening this distance, strong sufficient choice takes effect again. The only cases where this does not happen are that $ M $ is either a bull or a kite. The bull graph has choice variance $ \ch(Bull) = 1 $, and, this time, structural choice implies the induction hypothesis, and $ M $ can not be a kite graph as outlined in the proof of Lemma \ref{SCLStructuralChoice}. Finally, if $ M $ is spanned by diamond graph, when the next probe $ p^{+} $ appears at a girdle vertex of it, we arrive at Situation \ref{SCLMinDeg3RobberStrati}, whereas every other probe $ p^{+} $ leads to Situation \ref{SCLMinDeg3RobberStratiii} claimed in the induction hypothesis. \\

\item Next, let $ D = 2 $. This means that we have $ \Gamma_{H}(x) \cap \Gamma_{H}(y) \neq \varnothing $, but $ x \nsim y $. We thus have $ \diam(M) \leq 4 $ and the strong sufficient choice argument finishes the argument for this case if $ \vert M \vert \geq 7 $. Otherwise, $ \vert M \vert \leq 6 $. Then $ M $ must be spanned by $ K_{2,3} $, $ K_{2,4} $, a double-dart or a watch. In the first three cases, strong sufficient choice takes effect again. In the final case, $ \ch_{H}(M) = 1 $ and by structural choice, the induction statement follows as well. \\

\item In the final subcase we have $ D \geq 3 $, $ x $ and $ y $ are sufficiently far away to yield $ \Gamma_{H}(x) \cap \Gamma_{H}(y) = \varnothing $ and thus $ \vert M \vert \geq 8 $. Moreover, the diameter of each of the two neighbourhoods is at most $ 2 $, so $ \Var(M) \leq 6 $, implying $ \ch(M) \geq 2 $. Hence we are done by strong sufficient choice in this case again. \\



\end{enumerate}

\item The third and last case works under the assumptions that $ \vert L \vert = 2 $, $ p \in H $, $ d = 1 $ and $ \Gamma_{H}[L] $ spans a diamond graph, which equates to situation \ref{SCLMinDeg3RobberStratiii}. Write again $ L = \{ x, y \} $. 


In this case, clearly $ x $ and $ y $ are the girdle vertices of the diamond, and thus, by the assumption, both tip vertices have degree at least $ 4 $, i.e. without loss of generality have $ 2 $ neighbours outside the diamond each -- if the tip vertices were connected, $ H $ would contain a $ K_{4} $, immediately implying non-locatability, so assume this is not the case. Moreover assume none of these neighbours coincide, as otherwise there would be a $ C_{5} $ inside $ H $, thus making $ H $ non-locatable as again. After his move, the robber can be in $ M = L \cup \{ z \} $ for the neighbour $ z \neq p $ of $ x $ and $ y $. It is easy to see that whenever the next probe $ p^{+} $ of the cop is not done at $ x $ or $ y $, we have $ d(p^{+},x) = d(p^{+},y) $, and this common distance can then be returned by the robber, resulting in $ \Gamma_{H}[L^{+}] $ being the same diamond for the next set of possibly locations $ L^{+} $ again. If the cop probes at $ x $ (or $ y $, by symmetry), the robber can return distance $ 1 $, potentially being in $ y $ or $ z $, resulting in $ \Gamma_{H}[L^{+}] $ being a diamond with at least two additional leaves on one of the tip vertices. The robber chooses all those additional leaves as the set $ L^{+} $, resulting in situation \ref{SCLMinDeg3RobberStrati} or \ref{SCLMinDeg3RobberStratii}.

\end{enumerate}

This shows that $ G $ is non-locatable, thus finishing the proof.
\end{proof}

The authors note that due to the assumptions in Theorem \ref{SCLMinDeg3HideoutCondition}, $ \Gamma_{H}[L] $ can never have the shape of a kite graph. However, the result can be shown to still hold if the kite was included as a possible subgraph spanned by two vertices in Theorem \ref{SCLMinDeg3HideoutCondition}. The details are however time-consuming and offer no additional insight, so we leave this as an exercise to the (very) interested reader.


We would also like to point out that Theorem \ref{SCLMinDeg3HideoutCondition} allows us to deduce a non-trivial result about the locatability of random graphs, an area which as of yet is completely unexplored and deserves further conisderation. We can easily deduce the following result.

\begin{Corollary}
\label{SCLRandomCubic}
A random cubic graph is non-locatable with high probability.
\end{Corollary}

\begin{proof}
By an easy first moment argument, it can be seen that with high probability, a random cubic graph has no diamond subgraphs. Thus a random cubic graph with high probability trivially meets the condition of Theorem \ref{SCLMinDeg3HideoutCondition} and is thus a hideout graph.
\end{proof}

Restricting the set of cubic graphs under consideration appropriately might well result in a smaller family of graphs for which this is no longer the case, and finding such restrictions remains an open problem which we pose again in the open problems section \ref{SCLOpenProblemsSection}.

The following example however shows that Theorem \ref{SCLMinDeg3HideoutCondition} is best possible in some sense, in that it does no longer hold if only one vertex per diamond subgraph is required to have higher degree. Moreover, this example, which we refer to as the \textit{pretzel graph}, shown in Figure \ref{SCLPretzelGraphFigure}, is a locatable but not $ 3 $-colourable graph, thus also disproving our previous conjecture that all locatable graphs are $ 3 $-colourable. We do not claim that this is the smallest such example. 


\begin{figure}[h!] \centering
  \begin{tikzpicture}
	\tikzstyle{vertex}=[draw, shape=circle, minimum size=5pt, inner sep=0]
	\foreach \centrex/\centrey/\letter/\angle in {1.37/0/B/270, -2.91/-0.53/D/307.9, 2.73/-2.06/E/270, 0.00/-2.06/F/270, -2.73/-2.06/G/270}
	{\foreach \dist/\baseangle/\number in {0.87/90/1, 0.50/180/2, 0.5/0/3, 0.87/270/4}
	 {\pgfmathsetmacro{\xdist}{(\centrex + \dist*cos(\angle + \baseangle)}
		\pgfmathsetmacro{\ydist}{\centrey + \dist*sin(\angle + \baseangle)}
		\node[vertex](\letter\number) at (\xdist, \ydist) {};
	  }
		 }
	\foreach \letter in {B, D, E, F, G}
	{\foreach \from/\to in {1/2, 1/3, 2/3, 2/4, 3/4}
	 {\draw (\letter\from) -- (\letter\to);
	  }
	   }
	\foreach \centrex/\centrey/\letter/\angle in {2.91/-0.53/A/232.1, -1.37/0.00/C/270, 4.46/-2.43/H/0, 3.72/-3.74/I/301.1, 2.15/-3.93/J/252.8, 0.66/-3.12/K/230, -4.46/-2.43/L/0, -3.72/-3.74/M/58.9, -2.15/-3.93/N/107.2, -0.66/-3.12/O/130}
	{\foreach \dist/\baseangle/\number in {0.87/90/1, 0.50/180/2, 0.5/0/3}
	 {\pgfmathsetmacro{\xdist}{(\centrex + \dist*cos(\angle + \baseangle)}
		\pgfmathsetmacro{\ydist}{\centrey + \dist*sin(\angle + \baseangle)}
		\node[vertex](\letter\number) at (\xdist, \ydist) {};
	  }
		 }
	\foreach \letter in {A, C, H, I, J, K, L, M, N, O}
	{\foreach \from/\to in {1/2, 1/3, 2/3}
	 {\draw (\letter\from) -- (\letter\to);
	  }
	   }
	\foreach \from/\to in {A/B, C/D, H/I, I/J, J/K, L/M, M/N, N/O}
	{\draw (\from2) -- (\to1);
	 \draw (\from3) -- (\to1);
		}
	\foreach \from/\to in {A1/E1, A1/H1, E1/H1, D4/G4, D4/L1, G4/L1, B4/C1, E4/F1, F4/G1, K2/F3, K3/F3, O2/F3, O3/F3}
	{\draw (\from) -- (\to);
	 }
	\node at (-3.90, -1.56) {A};
	\node at (3.90, -1.56) {B};
	\node at (-4.46, -2.75) {C};
	\node at (4.46, -2.75) {D};
	\node at (-0.9, -1.8) {a};
	\node at (0.9, -1.8) {b};
	\node at (0.0, -3.0) {c};
	\node at (0.0, -1.3) {d};	
	\node at (0.0, 0.15) {e};
	\node at (-3.0, -4.5) {r};
	\node at (2.4, 0.2) {s};
	\node at (-4.7, -1.55) {u};
	\node at (-3.8, -1.0) {v};
	\node at (-3.25, -2.05) {w};
	\node at (4.7, -1.55) {x};
	\node at (3.8, -1.0) {y};
	\node at (3.25, -2.05) {z};
  \end{tikzpicture}
  \caption{The pretzel graph, an example for a locatable, not $ 3 $-colourable graph, with some labels used in the proofs of theorems \ref{SCLPretzel4Col} and \ref{SCLPretzelLoc}.}
  \label{SCLPretzelGraphFigure}
\end{figure}
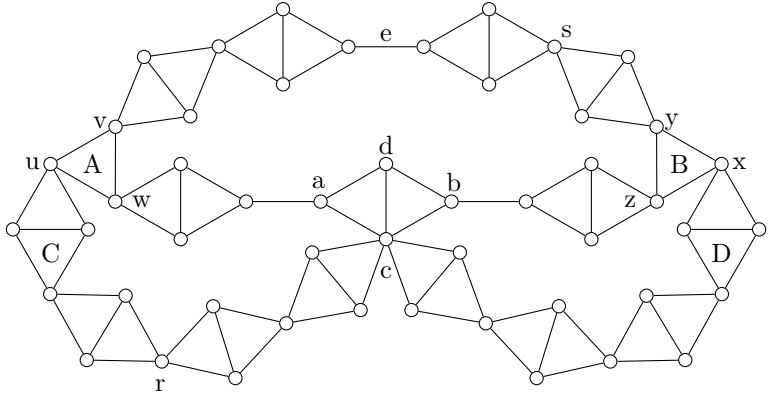

We will now show that the pretzel graph is indeed a valid counterexample to our original conjectureby proving the two properties we claim for it above, starting with a lower bound on the chromatic number -- the fact that $ 4 $ colour indeed suffice to colour it is easy to see (or follows from Theorem \ref{SCLColourability} and Theorem \ref{SCLPretzelLoc}).

\begin{Theorem}
\label{SCLPretzel4Col}
The pretzel graph is not $ 3 $-colourable.
\end{Theorem}

\begin{proof}
We will attempt to colour the pretzel graph in $ 3 $ colours, and will fail doing so. Note that when attempting to colour the graph with $ 3 $ colours, the two tip vertices of any given diamond subgraph must have the same colour. Start by colouring the unique vertex of degree $ 7 $, $ c $, red. Thus the tip vertices of all diamonds on the two chains of diamonds originating in $ c $ must then be red, which includes one vertex each on the two triangles $ A $ and $ B $, namely $ u $ and $ x $. None of the other vertices on $ A $ or $ B $ can thus be red. Now assume $ v $ and $ y $ have the same colour, blue say. This results in giving both vertices of the edge $ e $ the same colour, blue, and thus leads to a contradiction. Thus one of $ v $ and $ y $ is green, $ v $ say, the other one is blue, from which we can deduce that $ w $ and $ z $ also have distinct colours, the former being blue, the latter being green. This results in the two tip vertices of the central diamond, $ a $ and $ b $, being unable to have the same colour, as $ a $ cannot be blue, $ b $ cannot be green and none of them can be red, giving a contradiction, as we can then not colour $ d $ in any of the three colours.
\end{proof}

Next, we will show that the pretzel graph is locatable. To make the proof more accessible, we break the cop strategy down into multiple lemmas, and then show how the result can be deduced from this. To make the argument easier to follow, it is possible to imagine a larger pretzel graph instead, where each of the chains of diamonds (of which there are $ 6 $) can be made as long as desired, without changing the validity of the colourability argument above. In what follows, a \emph{chain of diamonds} within the pretzel graph is a connected subgraph of the pretzel graph consisting only of a path of diamond graphs, each connected to those adjacent to it through common tip vertices, and with no vertex of degree higher than $ 4 $ in the original graph. The \emph{end} of a chain of diamonds is a vertex of degree $ 2 $ in the subgraph. The following lemma shows that the robber can always be forced to move along a certain direction in such a chain of diamonds.

\begin{Lemma}
\label{SCLDiamondChainPush}
If the robber is know to be inside a chain of diamonds, he can be pushed out of the chain through either end of it. 
\end{Lemma}

\begin{proof}
Assume the robber is inside a chain of diamonds. Let the cop probe one of its ends. This will either locate the robber or force him into a set of two neighbouring girdle vertices of some diamond. After his next move, the robber is in that diamond. 
Probing this diamond's tip vertex at the opposite end from the direction into which the cop wishes to push the robber forces the latter into the remaining three vertices of the same diamond (after moving), after which a probe at one of the girdle vertices forces him to claim to be in one of the two remaining vertices, all by the no-backtrack condition. The neighbourhood of this set has the shape of a path of length $ 4 $ with one pending edge, and a probe at the far end from the pending edge finally forces the robber into the next diamond, from where this cop strategy can be repeated. After finitely many moves, the robber has been pushed into the last diamond of the chain, and repeating the strategy once more, the cop can force him to leave this chain.
\end{proof}

A simple corollary show us which chains of diamonds will result in the robber being caught by the above lemma.

\begin{Corollary}
\label{SCLPretzelLocatableChains}
If the robber is known to be in a chain of diamonds one end of which is $ v $, $ w $, $ y $ or $ z $, the cop can locate him.
\end{Corollary}

\begin{proof}
By pushing the robber away from $ v $, $ w $, $ y $ or $ z $ and through the other end of the chain of diamonds he is in, he will be forced into a geodesic path, on which he can be located.
\end{proof}

The next lemma captures the essence of what happens when the robber is forced to move into either of the triangles $ A $ or $ B $, and how the cop can control the direction in which he leaves it again, allowing her to capture the robber. Without loss of generality by symmetry, we concentrate on him moving into $ B $. It is actually sufficient for our purposes to consider a slightly weaker condition, assuming the robber to start in the slightly larger set $ N[y,z] $.

\begin{Lemma}
\label{SCLPretzelTrianglesCopGood}
If the robber is known to be inside $ N[y,z] $, he can be located.
\end{Lemma}

\begin{proof}
Assume the robber is inside $ N[y,z] $. The cop than probes $ x $, upon which the robber has two possible replies. If he replies $ 2 $, a follow-up probe at $ b $ will force him to commit to being in one of the two chains as required. If he replies $ 1 $, thus being in $ N[y,z] \setminus {x} $, the cop then probes $ z $. Again, possible answers by the robber are $ 1 $ and $ 2 $, the latter forcing him into the chain of diamonds ending in $ y $. If he replies $ 1 $, a follow-up probe in $ x $ forces him to reply $ 2 $ or be located, and as before, a probe at $ b $ then determines which chain of diamonds he moved into. In any case, the robber is known to be in a chain of diamonds ending in $ y $ or $ z $, allowing the cop to locate him by Corollary \ref{SCLPretzelLocatableChains}. 
\end{proof}

We now turn our attention to the core of the cop's strategy, with the only remaining difficulty occuring when the robber stays near the central degree-$ 7 $-vertex $ c $.

\begin{Theorem}
\label{SCLPretzelLoc}
The pretzel graph is locatable.
\end{Theorem}

\begin{proof}
We describe the cop's winning strategy. Let her first probe be at $ c $. Should the robber return distance $7 $, $ 8 $ or $ 9 $, with $ 9 $ being the largest possible result, he must be in the two chains of diamonds comprising of four diamonds and the single edge $ e $ lying between $ v $ and $ y $. One more probe reveals which chain he is in, allowing the robber to locate him by Corollary \ref{SCLPretzelLocatableChains}. \\
If the robber returns distance $ 6 $ after the first probe, he can be in any vertex of $ C $ or $ D $, or the neighbours of $ v $ and $ y $ towards $ e $, and move to the neighbourhood of this set of vertices. Following probes at $ r $, and then if necessary $ z $ and $ s $ (as one of many possibilities), force the robber into a chain of diamonds, allowing the cop to chase him along the chain according to Lemma \ref{SCLDiamondChainPush}. This can happen either towards $ e $, resulting in capture by Corollary \ref{SCLPretzelLocatableChains}, or towards one of the triangles $ A $ and $ B $ when he is in a chain ending in $ c $. Pushing him into either of those, without loss of generality into $ B $ via $ x $, will force him into $ N[y,z] $, from where capture is possible by Lemma \ref{SCLPretzelTrianglesCopGood}. \\
If the robber returns distance $ 2 $, $ 3 $, $ 4 $ or $ 5 $ upon the first probe, the set of his possible locations afterwards is a little bit more complicated, and results in four disconnected components. Probing in one of those appropriately, similar to how it was done when the robber's first reply was $ 6 $, the cop can reduce the number of components in which the robber could end up to $ 1 $. Noting that this one component will always be a subset of a chain of diamonds, Lemmas \ref{SCLDiamondChainPush} and \ref{SCLPretzelTrianglesCopGood} together with Corollary \ref{SCLPretzelLocatableChains}, allow the cop to locate the robber similar to the case for distance $ 6 $. \\
If the robber hides in the neighbourhood of $ c $ on the first probe, i.e. returns distance $ 1 $, the cop's next probe can be done at $ r $, allowing him to determine whether the robber is in the chain of diamonds containing $ r $ and, if so, pushing him towards $ A $. If not, the next probe at $ c $ makes sure that he does not reenter $ c $ and thus stays away from this chain - if he does not reply $ 1 $ to this probe, we arrive at a strictly easier situation than one already dealt with. If he once again returns $ 1 $, a probe at the mirror image of $ r $ on the opposite chain of diamonds, combined with a following probe at $ c $ again, allows the cop to either force him into that chain or rule out his presence in it in exactly the same way. The last probe at $ c $ either return $ 1 $ or a distance greater than $ 1 $, the latter again leading to situations strictly easier than ones already dealt with. But if the robber returns $ 1 $ after this last probe, the cop can only spread to a geodesic path, a probe at one end of which will locate him. \\
This finishes the cop's strategy and allows for capture in every possible situation. 
\end{proof}

\section{Conclusion and Open Problems}
\label{SCLOpenProblemsSection}

We will conclude this paper with a few open problems. Let us start by posing an open problem regarding the robber's winning strategy for results showing that a certain graph $ H $ is a hideout graph, which is the main type of result in section \ref{SCLForbiddenSubgraphsDiam2Section} and partially section \ref{SCLColourabilitySection}. In all those proofs, as outlined repeatedly in the induction steps, the robber restricts himself to movement on the subgraph $ H $. It is not clear, though, that this strategy does always work for these kind of results. 
As this question strikes us a very interesting and important one, we will state it as an open problem.

\begin{Open}
\label{SCLHideoutsClosedQuestion}
Let $ H $ be a graph, such that no graph containing $ H $ as an (induced) subgraph is locatable. Does there always exist a winning strategy for the robber where he does not leave $ H $?
\end{Open}

We do believe that, should graphs $ H $ exist for which this strategy does not work, the proof that any graph containing $ H $ is non-locatable would have to be non-constructive. This is because if the robbers winning strategy requires him to leave $ H $, we have no chance of laying out his general winning strategy, as we have no knowledge of the structure in $ G $ outside $ H $. On the other hand, it might be the case that the strategy of restricting himself to $ H $ will always work for the robber. 

Another area that deserves further attention is that of finding additional hideout graphs, or weak hideout graphs for that matter, ideally giving a full characteration.

\begin{Open}
\label{SCLHideoutCharacterisationQuestion}
Find a complete characterisation of all hideout graphs, respectively weak hideout graphs.
\end{Open}

Failing this, one can ask simpler questions that can then pave the way towards a deeper understanding of hideout graphs. For example, it would be interesting to figure out how common -- or rare -- hideout graphs are. Here, we are interested in minimal hideout or weak hideout graphs, i.e. those for which any subgraph is locatable. We have already described one infinite such family of weak hideout graphs, the sunlets, in section \ref{SCLNoForbiddenCharacterisationSection}. We mention here another one infinite family, consisting of minimal hideout graphs and thus being the first such family. These are the diamond rings -- the diamond ring of order $ n $ consisting of a cycle $ C_{n} $ with each edge replaced by a diamond. The arguments from section \ref{SCLColourabilitySection} can be used to show non-locatability and minimality, and the hideout property is easily deduced by showing that whenever a graph contains a diamond ring the robber can still follow the same strategy directly in the diamond ring itself.

Finally, we have noted above in Corollary \ref{SCLRandomCubic} that random cubic graphs with high probability are non-locatable. It is thus interesting to ask whether we can restrict the family of cubic graphs sufficiently such that a random graph sampled from the remaining family of graphs asymptotically has a positive probability of being locatable. 

\begin{Open}
\label{SCLCubicSubfamilyLocatableQuestion}
Find a (sensibly) restricted subfamily of the family of cubic graphs, the members of which are locatable with positive probabiliy.
\end{Open}

\section{Acknowledgements}
\label{SLCAcknowledgementsSection}
The first author acknowledges support through funding from NSF grant DMS~1301614 and MULTIPLEX grant no. 317532, and is grateful to the organisers of the $8^\text{th}$ Graduate Student Combinatorics Conference at the University of Illinois at Urbana-Champaign for drawing his attention to the problem. The second author acknowledges support through funding from the European Union under grant EP/J500380/1 as well as from the Studienstiftung des Deutschen Volkes. The authors would also like to thank Yuval Peres and the Theory Group at Microsoft Research Redmond for hosting them while some of this research was conducted.

\end{document}